\documentclass[12pt]{article}
\usepackage{amsmath,amssymb,amsthm}

\newtheorem{theorem}{Theorem}

\newtheorem{corollary}[theorem]{Corollary}

\newtheorem{proposition}[theorem]{Proposition}

\newtheoremstyle{named}{}{}{\itshape}{}{\bfseries}{.}{.5em}{\thmnote{#3's }#1}
\theoremstyle{named}

\newtheoremstyle{nnamed}{}{}{\itshape}{}{\bfseries}{.}{.5em}{\thmnote{#3' }#1}
\theoremstyle{nnamed}

\newcommand{\CD}{\mathcal{C}\mathcal{D}}

\def\barr{\begin{array}}
\def\earr{\end{array}}

\title{Two classes of finite groups whose Chermak-Delgado lattice is a chain of length zero}
\author{Ryan McCulloch and Marius T\u arn\u auceanu}
\date{January 21, 2018}

\begin{document}

\maketitle

\begin{abstract}
    It is an open question in the study of Chermak-Delgado lattices precisely which finite groups $G$ have the property that $\CD(G)$ is a chain of length $0$.  In this note, we determine two classes of groups with this property.  We prove that if $G=AB$ is a finite group, where $A$ and $B$ are abelian subgroups of relatively prime orders with $A$ normal in $G$, then the Chermak-Delgado lattice of $G$ equals $\{AC_B(A)\}$, a strengthening of earlier known results.
\end{abstract}

{\small
\noindent
{\bf MSC2000\,:} Primary 20D30; Secondary 20D60, 20D99.

\noindent
{\bf Key words\,:} Chermak-Delgado measure, Chermak-Delgado lattice, Chermak-Delgado subgroup, subgroup lattice, ZM-group, dihedral group, Frobenius group.}

\section{Introduction}

Throughout this paper, $G$ will denote a finite group.  The relation between the structure of a group and the structure of its lattice of subgroups constitutes an important domain of
research in group theory. The topic has enjoyed a rapid development starting with the first half of the 20th century. Many classes
of groups determined by different properties of partially ordered subsets of their subgroups (especially lattices of (normal) subgroups)
have been identified. We refer to Schmidt's book \cite{11} for more information about this theory.

An important sublattice of the subgroup lattice $L(G)$ of a finite group $G$ is defined as follows. Denote by
\begin{equation}
m_G(H)=|H||C_G(H)|\nonumber
\end{equation}the \textit{Chermak-Delgado measure} of a subgroup $H$ of $G$ and let
\begin{equation}
m(G)={\rm max}\{m_G(H)\mid H\leq G\} \mbox{ and } {\cal CD}(G)=\{H\leq G\mid m_G(H)=m(G)\}.\nonumber
\end{equation}Then the set ${\cal CD}(G)$ forms a modular self-dual sublattice of $L(G)$,
which is called the \textit{Chermak-Delgado lattice} of $G$. It was first introduced by Chermak and Delgado \cite{5}, and revisited by Isaacs \cite{9}. In the last years
there has been a growing interest in understanding this lattice (see e.g. \cite{1,2,3,6,10,14}).

Most current study has focused on $p$-groups.  Groups whose Chermak-Delgado lattice is a chain or a quasi-antichain have been studied.  It is easy to see that if $G$ is abelian, then $\CD(G) = \{ G \}$ is a chain of length $0$.  An open question is to classify those groups $G$ for which $\CD(G)$ is a chain of length $0$.  In this paper we determine two classes of groups which have this property.  Section 1 concerns semidirect products of abelian groups of coprime orders, and Section 2 concerns Frobenius groups.

Recall two important properties of the Chermak-Delgado lattice that will be used in our paper:
\begin{itemize}
\item[$\cdot$] if $H\in {\cal CD}(G)$, then $C_G(H)\in {\cal CD}(G)$ and $C_G(C_G(H))=H$;
\item[$\cdot$] the minimum subgroup $M(G)$ of ${\cal CD}(G)$ (called the \textit{Chermak-Delgado subgroup} of $G$) is characteristic, abelian, and contains $Z(G)$.
\end{itemize}

\section{Semidirect Products of Abelian Groups of Coprime Orders}

The following appears in \cite{8}.

\begin{theorem}
    Let $N$ be nilpotent and suppose that it acts faithfully on $H$, where $(|N|,|H|)=1$.  Let $\pi$ be a set of prime numbers containing all primes for which the Sylow subgroup of $N$ is nonabelian.  Then there exists an element $x \in H$ such that $|C_N(x)|$ is a $\pi$-number, and if $\pi$ is nonempty, $x$ can be chosen so that $|C_N(x)| \leq (|N|/p)^{1/p}$, where $p$ is the smallest member of $\pi$.
\end{theorem}

Taking $\pi = \emptyset$ we obtain:

\begin{corollary}
Let $N$ be abelian and suppose that it acts faithfully on $H$, where $(|N|,|H|)=1$.  Then there exists an element $x \in H$ such that $C_N(x)=1$.
\end{corollary}

We note that Corollary 2 has an elementary proof using a variation on Brodkey's Theorem, see \cite{4}.

\begin{theorem}
    Let $G$ be a finite group which can be written as $G=AB$,
    where $A$ and $B$ are abelian subgroups of relatively prime
    orders and $A$ is normal. Then
    \begin{equation}
    m(G)=|A|^2|C_B(A)|^2 \mbox{ and }   {\cal CD}(G)=\{AC_B(A)\}.\nonumber
    \end{equation}
\end{theorem}

\begin{proof}
Let $M = M(G)$.  Since $M$ is abelian and $M \cap B$ is a Hall $\pi$-subgroup of $M$, we have that $M \cap B$ is characteristic in $M$.  And since $M$ is a characteristic subgroup of $G$, we have that $M \cap B$ is a characteristic subgroup of $G$.   Thus $M \cap B$ centralizes $A \unlhd G$, and so $M \cap B \leq C_B(A)$.  Since $Z(G) \leq M$, we have that $C_B(A) = Z(G) \cap B \leq M \cap B$.  Thus $M \cap B = C_B(A)$.

And so $M = UC_B(A)$ for some $U \leq A$.  It suffices to show that $U = A$, as $AC_B(A)$ being self-centralizing would imply that $\CD(G) = \{ AC_B(A) \}$.

Let $K$ be the kernel of the action of $C_B(U)$ on $A/U$.  Let $p$ be an arbitrary prime, and let $P$ be a Sylow $p$-subgroup of $K$.  Let $a \in A$ be arbitrary.  Consider the action of $P$ on elements of the coset $aU$.    Note that either $P$ centralizes $aU$, or every element of $aU$ has a nontrivial orbit under $P$.  This is true because $P \leq C_B(U)$.  But in the latter case, being that $P$ is a $p$-group, this would imply that $p$ divides $|aU| = |U|$ which is coprime to $p$.  Thus, it must be the former case that $P$ centralizes $aU$.  And since $a \in A$ was arbitrary, and since $A$ is the union of its $U$-cosets, we have that $P$ centralizes $A$.  Now, $p$ was arbitrary, and $K$, being abelian, is a direct product of its Sylow subgroups, and it follows that $K$ centralizes $A$.  Thus $K = C_B(A)$.  And so $C_B(U)/C_B(A)$ acts faithfully on $A/U$.

By Corollary 2, an element of $A/U$ has a regular orbit under the action of $C_B(U)/C_B(A)$.  And so
\begin{equation}
|A:U|\geq |C_B(U):C_B(A)|.\nonumber
\end{equation}This leads to $m_G(AC_B(A))\geq m_G(M)$ and thus $m_G(AC_B(A))=m_G(M)$
by the maximality of $m_G(M)$. Then
\begin{equation}
|U||C_B(U)|=|A||C_B(A)|,\nonumber
\end{equation}implying that
\begin{equation}
|U|=|A|\mbox{ and } |C_B(U)|=|C_B(A)|\nonumber
\end{equation}because $A$ and $B$ are of coprime orders.  Hence $U=A$, as desired.
\end{proof}

Recall that a ZM-group is a finite group with all
Sylow subgroups cyclic. Zassenhaus classified such groups (see \cite{7}, \cite{15} for a good exposition).  They are of type:
\begin{equation}
{\rm ZM}(m,n,r)=\langle a, b \mid a^m = b^n = 1,
\hspace{1mm}b^{-1} a b = a^r\rangle, \nonumber
\end{equation}
where the triple $(m,n,r)$ satisfies the conditions
\begin{equation}
{\rm gcd}(m,n)={\rm gcd}(m,r-1)=1 \mbox{ and } r^n
\equiv 1 \hspace{1mm}({\rm mod}\hspace{1mm}m). \nonumber
\end{equation}Note that $|{\rm ZM}(m,n,r)|=mn$ and $Z({\rm ZM}(m,n,r))=\langle b^d\rangle$, where $d$ is the multiplicative order of
$r$ modulo $m$, i.e. $d={\rm min}\{k\in\mathbb{N}^* \mid r^k\equiv 1 \hspace{1mm}({\rm mod} \hspace{1mm}m)\}$. The structure of
${\cal CD}({\rm ZM}(m,n,r))$ has been determined in the main theorem of \cite{12}, and can also be obtained from Theorem 3, taking $G={\rm ZM}(m,n,r)$, $A=\langle a\rangle$, and $B=\langle b\rangle$. Thus, Theorem 3 generalizes \cite{12}.

\begin{corollary}
   Let $G = ZM(m,n,r)$ be a $ZM$-group.  Then
    \begin{equation}
    m(G)=\frac{m^2n^2}{d^2} \mbox{ and }
    {\cal CD}(G)=\{\langle a, b^d\rangle\}.\nonumber
    \end{equation}
\end{corollary}

There are many other classes of groups which fall under the umbrella of Theorem 3.  A natural one occurs as follows.  Let $V$ be a vector space and let $H$ be an abelian subgroup of $Aut(V)$ so that $H$ and $V$ have coprime orders, and let $G = V \rtimes H$ be the semidirect product.  Then by Theorem 3 we have that $\CD(G) = \{ V \}$.

\section{Frobenius Groups}

If $A$ acts on a group $N$ we say that the action is \textbf{Frobenius} if $n^a \neq n$ whenever $n \in N$ and $a \in A$ are nonidentity elements.  The following results on Frobenius actions can be found in \cite{9}.

\begin{proposition}\label{prop1.12}
Let $N$ be a normal subgroup of a finite group $G$, and suppose that $A$ is a complement for $N$ in $G$.  The following are then equivalent.
\begin{enumerate}
\item The conjugation action of $A$ on $N$ is Frobenius.
\item $A \cap A^g = 1$ for all elements $g \in G \setminus A$.
\item $C_G(a) \leq A$ for all nonidentity elements $a \in A$.
\item $C_G(n) \leq N$ for all nonidentity elements $n \in N$.
\end{enumerate}
\end{proposition}

If both $N$ and $A$ are nontrivial in the situation of Proposition~\ref{prop1.12}, we say that $G$ is a \textbf{Frobenius group} and that $A$ and $N$ are respectively the Frobenius complement and the Frobenius kernel of $G$.

In 1959, J.G. Thompson proved that a Frobenius kernel is nilpotent, see \cite{13}.

\begin{theorem}\label{thm1.4}
Suppose $G = NA$ is a Frobenius group with kernel $N$ and complement $A$.  Then $\mathcal{C}\mathcal{D}(G) = \mathcal{C}\mathcal{D}(N)$.
\end{theorem}

\begin{proof}
First we argue that $1$ and $G$ are not in $\CD(G)$.  Since $G$ is Frobenius, we have $Z(G) = 1$.  On the other hand, since $|N| \equiv 1 \bmod |A|$ and $N$ is nontrivial, we have $|N| > |A|$.  By Thompson's theorem, Frobenius kernels are nilpotent, and so $Z(N) > 1$.  Now the action of $A$ on $Z(N)$ is also Frobenius, therefore $|Z(N)| \equiv 1 \bmod |A|$.  The implies $|Z(N)| > |A|$, and thus
$$m_G(N) = |N||Z(N)| > |N||A| = |G| = m_G(G) = m_G(1).$$

Consequently, $1$ and $G$ are not in $\CD(G)$.

Let $T$ be the greatest element of $\CD(G)$.  If $T \cap N = 1$, then since $T \unlhd G$, we have that $T \leq C_G(N) \leq N$.  So $T = 1$, a contradiction.  Therefore we have $T \cap N > 1$.  Let $1 \neq n \in T \cap N$.  Then $C_G(T) \leq C_G(T \cap N) \leq C_G(n) \leq N$.  Now $C_G(T) \in \CD(G)$, implying that $C_G(T) \neq 1$.  Thus, $T = C_G(C_G(T)) \leq N$.  Hence $\CD(G) = \CD(T) = \CD(N)$.
\end{proof}

By Theorem 6, we infer that the class of Frobenius groups with abelian kernel provides another class of groups $G$ for which $\CD(G)$ is a chain of length 0.

\begin{proposition}\label{prop7}
Let $P$ be a $p$-group which possesses an abelain subgroup $A$ so that $|P : A| = p$.  Then $\CD(P) = \{ A \}$ if and only if $|P : Z(P) | > p^2$.
\end{proposition}

\begin{proof}
Let $|P| = p^n$.  Then $|A| = p^{n-1}$.  Suppose $\CD(P) = \{A\}$.  $P$ is non-abelian as otherwise $\CD(P) = \{P\}$.  Note that $C_P(A) = A$.  This is true because otherwise $C_P(A) = P$ implies that $A = Z(P)$ and then $P/Z(P)$ is cyclic, i.e. $P$ is abelian, a contradiction.  Thus, $m_P(A) = |A|^2 = p^{2n-2}$.  Now $P$ is not in $\CD(P)$, and so $|P||Z(P)| < m_P(A)=  p^{2n-2}$.  This implies that $|P:Z(P)| > p^2$.

Conversely, suppose that $|P:Z(P)| > p^2$.  Then $P$ is non-abelian, and as before, we conclude that $C_P(A) = A$.  Since $|P:Z(P)| > p^2$, it follows that $|P||Z(P)| < p^{2n-2}$, and so $m_P(P) < p^{2n-2} =  m_P(A)$, i.e. $P$ is not in $\CD(P)$.  Let $T$ be the largest member of $\CD(P)$.  If $T$ were to have measure larger than $A$, then $T$ must have measure $p^{2n-1}$, and it follows that $|T| = p^{n-1}$ and $C_P(T) = P$, but as before, this would imply that $P$ is abelian, a contradiction.  Thus, $T$ has the same measure as $A$.  Since $A$ is maximal in $P$, we obtain $T = A$.  But $A$ is abelian and is self-centralizing, there it is also the least member of $\CD(P)$.  Consequently, $\CD(P) = \{A\}$.
\end{proof}

We end with an example of a Frobenius group, $G$, with a non-abelian kernel, so that $\CD(G)$ is a chain of length $0$.

\bigskip\noindent{\bf Example.} Let
$$P=\left\{\left(
    \begin{array}{ccc}
    1 & a & b \\
    0 & 1 & c \\
    0 & 0 & 1 \\
    \end{array}
    \right) \mid a,b\in GF(7^2), c\in GF(7)\right\},$$where $GF(7)$ is seen as a subgroup of $GF(7^2)$. Then $P$ is a non-abelian group of order $7^5$. Clearly, its center
$$Z(P)=\left\{\left(
    \begin{array}{ccc}
    1 & 0 & b \\
    0 & 1 & 0 \\
    0 & 0 & 1 \\
    \end{array}
    \right) \mid b\in GF(7^2)\right\}$$is of order $7^2$ and so $|P : Z(P)| = 7^3 > 7^2$. We remark that $P$ has an abelian subgroup
$$A=\left\{\left(
    \begin{array}{ccc}
    1 & a & b \\
    0 & 1 & 0 \\
    0 & 0 & 1 \\
    \end{array}
    \right) \mid a,b\in GF(7^2)\right\}$$of index $7$ in $P$. By Proposition 7, we have that $\CD(P)=\{A\}$.

Define an automorphism $x$ of $P$ such that $${\left(
    \begin{array}{ccc}
    1 & a & b \\
    0 & 1 & c \\
    0 & 0 & 1 \\
    \end{array}\right)}^x=\left(
    \begin{array}{ccc}
    1 & 2a & 4b \\
    0 & 1 & 2c \\
    0 & 0 & 1 \\
    \end{array}\right).$$Then $x$ is fixed-point free of order 3 and so the semidirect product $G = P \rtimes \langle x\rangle$ is a Frobenius group of order $3\cdot 7^5$ with Frobenius kernel $P$, and by Theorem 6 we have $$\CD(G)=\CD(P)=\{A\},$$as desired.

\bigskip\noindent {\bf Acknowledgements.} The authors are grateful to Professor I.M. Isaacs for providing the statement of Theorem 3 and an outline of the proof.

\vspace*{3ex}
\small

\begin{minipage}[t]{7cm}
Ryan McCulloch\\
Assistant Professor of Mathematics\\
University of Bridgeport\\
Bridgeport, CT 06604\\
e-mail: {\tt rmccullo@bridgeport.edu}
\end{minipage}
\hfill
\begin{minipage}[t]{7cm}
Marius T\u arn\u auceanu \\
Faculty of  Mathematics \\
``Al.I. Cuza'' University \\
Ia\c si, Romania \\
e-mail: {\tt tarnauc@uaic.ro}
\end{minipage}

\end{document}